\documentclass[11pt,english]{smfart}
\usepackage[T1]{fontenc}
\usepackage[french,english]{babel}
\usepackage{setspace}
\usepackage{layout}
\usepackage{amstext,amssymb,amscd}
\usepackage{smfthm}

\def\C{{\mathbb C}}

\def\Q{{\mathbb Q}}

\def\R{{\mathbb R}}

\def\Z{{\mathbb Z}}
\def\Ga{{\Gamma}}

\def\H{{\mathcal{H}}}
\def\L{{\mathcal L}}
\def\O{{\mathcal O}}

\def\bk{{\backslash}}

\def\Hom{\text{Hom}}

\def\i{\iota}

\def\Ga{{\Gamma}}
\def\om{{\omega}}

\def\g{{\mathfrak g}}
\def\h{{\mathfrak h}}
\def\k{{\mathfrak k}}
\def\l{{\mathfrak l}}

\def\p{{\mathfrak p}}

\def\q{{\mathfrak q}}
\def\r{{\mathfrak r}}
\def\t{{\mathfrak t}}
\def\u{{\mathfrak u}}

\def\Gn{{G}}
\def\G{\underline{G}}
\def\Hn{{H}}
\def\H{\underline{H}}

\author{Mathieu Cossutta}
\address{EPFL SB IMB TAN
MA C3 604 (B\^atiment MA)
Station 8
CH-1015 Lausanne}
\email{mathieu.cossutta@ens.fr}
\title{Automorphic Lefschetz properties via $L^2$ cohomology}

\begin{document}
\renewcommand{\labelitemi}{$\bullet$}
\frontmatter
\begin{abstract}
In this paper one proves a special case of a conjecture by Nicolas Bergeron
\cite[conjecture 3.14]{B2}. This conjecture is a kind of automorphic Lefschetz property. It
relates the primitive cohomology of a locally symmetric
manifolds modeled on $U(p,q+r)$ to the primitive cohomology of some of its totally
geodesic submanifolds that are locally symmetric and modeled on $U(p,q)$. 
\end{abstract}
\maketitle

\section{Introduction}\label{sect1}

Let $\Gn$ be a connected reductive Lie group of compact center and $K$ be a maximal
compact subgroup of $\Gn$. The quotient
$$X_G=\Gn/K$$
is the symmetric space associated to $G$, let $d_G$ be its dimension. It is
naturally a Riemannian manifold on which $\Gn$ acts by isometries. For
$\Gamma$ a discrete subgroup of $G$, one defines the locally symmetric manifold
\[S_G(\Gamma)=\Gamma\bk X_G.\]
Let $\G$ be an anisotropic algebraic group defined over $\Q$ such that the
non-compact part of $\G(\R)$ is equal to $G$. Let $\rho: \G\rightarrow
\text{GL}(N)_{/\Q}$ be a closed immersion of algebraic groups. Let $n$ be a
non-negative integer, one defines
$$\Gamma(n)=\{\gamma\in G(\Q)|\ \rho(\gamma)\in \text{GL}(N,\Z)\text{ and
}\rho(\gamma)\equiv I_N\ [n]\}.$$
The group $\Gamma(n)$ is called a congruence subgroup and since $\G$ is
anisotropic, $\Gamma(n)$ is discrete and cocompact in $G$. 
One wants to study the link between the 
cohomology of  the compact  manifold
$S_G(\Gamma(n))=\Gamma(n)\bk X_G$ and the cohomology of some of its submanifolds. Let $H$ be a closed reductive subgroup of $G$ such that
\[ \Hn\cap K\text{ is a maximal compact subgroup of }\Hn.\]
This hypothesis implies that $X_H$ is a totally geodesic submanifold of $X_G$. Let $\H$ be an algebraic subgroup of $\G$
defined over $\Q$ such that the non-compact part of $H(\R)$ is equal to
$H$. One assumes furthermore that the inclusion of algebraic group
$\H\subset \G$ induces the inclusion of Lie groups $\Hn\subset \Gn$. One defines 
$$\Lambda(n)=\{\gamma\in H(\Q)|\ \rho(\gamma)\in \text{GL}(N,\Z)\text{ and
}\rho(\gamma)\equiv I_N\ [n]\}.$$
Since $\Lambda(n)=\Gamma(n)\cap H(\Q)$, there is a well defined natural natural map
$$j_{G,H,n}:S_H(\Lambda(n))\rightarrow S_G(\Gamma(n)).$$
This map is finite and according to  \cite[lemme principal et
  th\'eor\`eme 1]{B1}, there exists a finite index subgroup $\Gamma'$ of
$\Gamma(n)$ containing $\Lambda(n)$ such that the application 
$$j':S_H(\Lambda(n))\rightarrow S_G(\Gamma')$$
is an embedding. Let $i$ be a non-negative integer, we write
$$H^i(S_{\G},\C)=\varinjlim_{n}H^i(S_G(\Gamma(n)),\C).$$
The applications $j_{G,H,n}$ induce a direct image application
$$\left(j_{G,H}\right)_*:H^i(S_{\H},\C)\rightarrow H^{i+d_G-d_H}(S_{\G},\C).$$
We are interested in the case where $\Gn=U(p,q+r)$ and $\Hn=U(p,q)$
embedded in a standard way in $\Gn$. In this case the associated symmetric
spaces are Hermitian and the manifolds $S_G(\Gamma(n))$ and
$S_H(\Lambda(n))$ are projective. Using Matsushima formula (see equation \ref{eq2.3}), one can define
for two non-negative integers $i,j$ verifying $i+j\leq q$ (resp. $i+j\leq q+r$) a subspace
$$H^{ip,jp}(S_{\H})_{i,j}\  (\text{resp. } H^{ip,jp}(S_{\G})_{i,j})$$
of
$$H^{ip,jp}(S_{\H})\  (\text{resp. } H^{ip,jp}(S_{\G}))$$
(see definition \ref{defi2.3}). The cohomology classes of these subspaces are
primitive of be-degree $(ip,jp)$ (resp. $((i+r)p,(j+r)p$). They would be called highly primitive of type $(i,j)$. Except in the case where $p=1$,
being highly primitive is stronger than being primitive. For a cohomology
class, one can define its highly primitive part of type $(i,j)$ (see
defintion \ref{defi2.4}). We prove the following theorem.

\begin{theo}\label{theo0.1}
Let suppose that $p,q\geq 2$. Let $i,j$ be two non-negative integers such that
$i+j+r+1\leq q$ then the map
$$H^{pi,pj}(S_{\H},\C)_{i,j}\rightarrow H^{p(i+r),p(j+r)}(S_{\G},\C)_{(i+r),(j+r)}$$
obtained by projecting $\left(j_{G,H}\right)_*$ on the highly primitive cohomology
of type $(i+r,j+r)$ is injective.
\end{theo}

\begin{rema}
Let $i,j$ be natural numbers such that $i+j\leq q$. By a theorem of Li
(\cite[proposition 6.4]{Li1}), if we choose for $\H$ the automorphism group of an hermitian
form defined over a totally real numberfield, then if $p+q>2(i+j)$:
$$H^{ip,jp}(S_{\H},\C)_{(i,j)}\neq 0.$$
\end{rema}
Nicolas Bergeron proved in \cite[theorem 8.3]{B2} that it is enough to study a simpler problem. 

\begin{prop}\label{prop0.1}[\cite{B2} ] Let $i,j$ be two non-negative
  integers such that $i+j+r+1\leq q$ and assume that $p,q\geq 2$ then if
  for all congruence subgroups $\Lambda(n)$ the application
$$j_*: H^{ip,jp}(S_H(\Lambda(n)),\C)_{i,j}\rightarrow
H_2^{p(i+r),p(j+r)}(S_G(\Lambda(n)),\C)_{i+r,j+r}$$
obtained by projecting the direct image application in $L^2$-cohomology on
the highly primitive part of type $(i,j)$ is injective  then theorem \ref{theo0.1} is true. 
\end{prop}

\begin{rema}
We consider reduced $L^2$-cohomology. This means that for a Riemmannian
manifold $X$ and a non-negative integer $R$, $H_2^R(X)$ is the space of
$L^2$ harmonic differential forms of degree $R$ on $X$. If $r\neq 0$ the manifold $S_G(\Lambda(n))$ is non-compact then the $L^2$-cohomology can be different from the usual cohomology.
\end{rema}

 The hypothesis $i+j+r+1\leq q$ and $p,q\geq 2$ come from the proof of the
 proposition \ref{prop0.1}. The main theorem of this note is the following.

\begin{theo}\label{theo0.2} Let $i,j$ be natural integers such that
  $i+j+r\leq q$. Let  $\Lambda$ be a cocompact subgroup of $H$. The map
\begin{equation}\label{eq0.1} j_*: H^{ip,jp}(S_H(\Lambda),\C)_{i,j)}\rightarrow
H_2^{p(i+r),p(j+r)}(M_\Lambda,\C)_{i+r,j+r},
\end{equation}
obtained by projecting the direct image application on the highly primitive cohomology of type
$(i+r,j+r)$ is injective.
\end{theo}
The case where $i=j=0$, $H=U(1,1)$ and $G=U(2,1)$ was treated by Kudla and
Millson in \cite{KM2}. The general case where $i=j=0$ was done by Nicolas Bergeron in \cite[th\'eor\`eme
  3.4]{B2}. It is based on the article \cite{TW} of Tong and Wang. The
proof of theorem \ref{theo0.2} goes as follows (in fact some of the basic
ideas were developped by Kudla and Millson in the paper \cite{KM1} in the
case where $H=O(1,1)$ and $G=O(2,1)$). Let $\eta$ be
a highly primitive cohomology class of type $(i,j)$ on
$S_H(\Lambda)$. Since $S_G(\Lambda)$ can be seen as the
normal bundle of $S_H(\Lambda)$ in $S_G(\Gamma)$ (if $\Lambda=\Gamma\cap\Hn$), there exits a projection
$$p:S_G(\Lambda)\rightarrow S_H(\Lambda).$$
One can represent $j_*\eta$ (in $H^*(S_G(\Lambda))$) as the closed differential form
$$j_*\eta=p^*\eta\wedge[S_H(\Lambda)]$$
where $[S_H(\Lambda)]$  is a choice of differential form representing the dual class of $S_H(\Lambda)$ in
$S_G(\Lambda)$. One wants to choose $p^*\eta$ and $[S_H(\Lambda)]$ such that
$j_*\eta$ is harmonic and square-integrable. The representation theory of
$U(p,q)$ and $U(p,q+r)$ is used to make a choice that works. In the first part of
the note  some theorems of Nicolas Bergeron on cohomological
representations of $U(a,b)$ are recalled and in the second part the
theorem \ref{theo0.2} is proven. Finally, we remark that in order to generalize our
main results to other type of highly primitive cohomology or to other
groups one should prove
some theorems on restriction and tensor products of cohomological
representions generalizing theorem \ref{theoVI3} and \ref{theoVI4}.

\section{Representation theory}\label{sect2}

\subsection{Cohomological representations}
In this part $G$ will be equal to $U(a,b)$ viewed as the group of matrices
$$\left\{M\in M(n,\C)|\ \overline{^tM}I_{a,b}M=I_{a,b}\right\}$$
where $I_{a,b}=\text{diag}(I_a,-I_b)$. One can choose as a maximal compact
subroup of $G$,
$K=U(a)\times U(b)$ diagonaly embedded.
\begin{rema}
One uses the subscript $0$ for real Lie
algebra and no subscript for complex one. 
A compact Cartan algebra of both $\k_0$ and $\g_0$ is
$$\mathfrak{t}_0=
\left\{\text{diag}\left(x_1,\dots,x_a;y_1,\dots,y_b\right)|\ x_i,y_j\in\imath\R\right\}.$$
\end{rema}

Let $\p_0$ be the orthogonal complement for the Killing form of $\k_0$ in
$\g_0$. One has
$$\p=\left\{\left(
\begin{array}{cc}
0&A\\
B&0
\end{array}
\right)|\ A,\ ^tB\in M_{a,b}(\C)\right\}.$$
Since $\p_0$ is the tangent space at the identity of $X_G$, the Killing form (which is
positive and invariant by $K$ on $\p_0$) defines a Riemannian structure
on $X_G$. The group $G$ acts on it by isometries. Let $\Delta(\g,\t)$ be
the set of roots of $\t$ in $\g$ and $\mathfrak{g}^\tau$ be the eigenspace associated to a root
$\tau$. 

\begin{rema}\label{rema2.1} 
Since $K$ is compact, for all $H\in\imath\t_0$ and $\tau\in \Delta(\g,\t)$
the number $\tau(H)$ is real.
\end{rema}

Let $H\in\imath\mathfrak{t}_0$. One defines:
$$
\q(H)=\oplus_{\substack{\tau\in\Delta(\mathfrak{g},\mathfrak{t})\\\tau(H)\geq
0}}\ \mathfrak{g}^\tau,\ 
\l(H)=\oplus_{\substack{\tau\in\Delta(\mathfrak{g},\mathfrak{t})\\\tau(H)=0
}}\ \mathfrak{g}^\tau\text{ and }\ 
\u(H)=\oplus_{\substack{\tau\in\Delta(\mathfrak{g},\mathfrak{t})\\\tau(H)>0
}}\ \mathfrak{g}^\tau.
$$
Then $\q(H)$ is a parabolic algebra of $\g$ and $\q(H)=\l(H)\oplus \u(H)$ is a Levi
decomposition. Since $\l(H)$ is defined over $\R$, there exists a
well defined reductive subgroup $L(H)$ of $G$ of complexified Lie algebra $\l(H)$.
\begin{defi}\label{defi2.1}
A pair $(\q(H),L(H))$ defined by an element $H\in\imath\t_0$ is called a
theta stable parabolic algebra.
\end{defi}
Let $(\q,L)$ be a parabolic theta stable algebra of $\g$. Let $\u$ be the
radical unipotent of $\g$. One defines
$R(\mathfrak{q})=\dim\mathfrak{p}\cap\mathfrak{u}$, called the
cohomological degree of $\q$. According to Vogan and Zuckerman (\cite[theorem 2.5]{VZ}), $\bigwedge^{R(\mathfrak{q})}\left(\mathfrak{p}\cap\mathfrak{u}\right)$ is a
highest weight vector in $\bigwedge^{R(\q)}\mathfrak{p}$. Let
$V(\mathfrak{q})$ be the irreducible $K$-submodule of
$\bigwedge^{R(\mathfrak{q})}\mathfrak{p}$ generated by this vector. These
modules play an important role in the study of the cohomology of locally
symmetric spaces. One can classify them up to isomorphism. This is done for example by Bergeron in \cite{B2}. Clearly, if two theta stable parabolic algebras are
$K$-conjugated they generate the same module. So up to $K$-conjugation, we can
assume that $\q$ is defined by an element
$$H=(x_1,\dots,x_a)\otimes( y_1,\dots,y_b)\in \R^{a}\times\R^b$$
with
$$x_1\geq\dots\geq x_a\text{ and }y_1\geq\dots\geq y_b.$$
Such an element will be called dominant. One can associate to a dominant
element of $\imath\t_0$ two partitions. Recall that a partition is
a decreasing sequence $\alpha$ of natural integers
$\alpha_1,\dots,\alpha_l\geq 0$. The Young diagram of $\alpha$, also written $\alpha$, is obtained by adding from top to bottom rows of $\alpha_i$ squares all of
the same shape. Let $\alpha$ and $\beta$ be partitions such that the diagram of $\alpha$
is included in the diagram of $\beta$, one writes this relation $\alpha\subset\beta$. We
will also write $\beta\bk\alpha$ for the
complementary of the diagram of $\alpha$ in the diagram of $\beta$. It is
a skew diagram. One writes $a\times b$ or $b^a$ for the partition
$$(\underbrace{b,\dots,b}_{a\text{ times}}).$$ 
Let $H\in\imath\t_0$ be dominant. One associates to $H$ two partitions $\alpha\subset\beta\subset a\times b$ defined by:
\[
\alpha(i)=|\{j|x_i> y_{b+1-j}\}|\text{ and }\beta(i)= |\{j|x_i\geq
y_{b+1-j}\}|.
\]

\begin{prop}\label{prop2.1}  The following three points give the classification of modules $V(\q)$:
\begin{itemize}
\item let $\q$ be a parabolic theta stable algebra and $\alpha\subset \beta$ be the associated partitions, then $(\beta\bk\alpha)$ is an union of squares which intersect only on verteces.

\item we have $V(\q)=V(\q')$ in $\bigwedge^*\p$ if and only if $(\q,L)$ and $(\q',L')$ have the same associated
  partitions.
\item If $\alpha\subset\beta\subset a\times b$ is a pair of partitions verifying the
  condition of the first point there exists a parabolic theta stable algebra $\q$
  with the associated partition $(\alpha,\beta)$. Such a pair will be called compatible.
\end{itemize}
\end{prop}

\begin{defi} Let $(\alpha,\beta)$ be a pair of compatible partitions
  included in $a\times b$ and
  let $\q$ be a parabolic theta stable algebra of associated partitions
  $(\alpha,\beta)$. We will write $V_{\alpha,\beta}^{U(a,b)}=V(\q)$. This doesn't
  depend on the choice of $\q$ by the second point of the proposition \ref{prop2.1}.
\end{defi}

Parabolic theta stable algebras are related to representation theory of
$U(a,b)$ by cohomological induction (see \cite{KV} for definitions). The following theorem was
proven by Vogan and Zuckerman in \cite[theorem 2.5]{VZ}.

\begin{theo}\label{theo2.0} Let $\q$ be a parabolic theta stable
  algebra. There exists a unique irreducible and unitary representation of $G$, which will
be denoted  $A_\mathfrak{q}$, verifying the two following
properties :
\begin{itemize}
\item the infinitesimal character of $A_\q$ is the infinitesimal
  character of the trivial representation, 
\item the $K$-type $V(\q)$ appears in $A_\q$.
\end{itemize}
\end{theo}

\begin{defi} Let $(\alpha,\beta)$ be a pair of compatible partitions of
  $a\times b$. By
  unicity in \ref{theo2.0}, there exists a unique unitary representation
   that contains $V_{\alpha,\beta}^{U(a,b)}$ and of trivial
  infinitesimal character. One can write this representation
  $A_{\alpha,\beta}^{U(a,b)}$. Furthermore if $i,j$ are two non-negative integers
  such that $i+j\leq b$, one defines
 $V_{i,j}^{U(a,b)}=V_{(i^p,(b-j)^p)}$ and $A_{i,j}^{U(a,b)}=A_{(i^p,(b-j)^p)}$.

\end{defi}

\subsection{Some results of Nicolas Bergeron on $\bigwedge^*\p$}\label{ssect1.2}
The results of this section are mainly due to Nicolas Bergeron. One uses
notations of the introduction. For example, we have that $G=U(p,q+r)$ and $H=U(p,q)$.

\begin{theo}[\cite{B2} lemma 2.3 and theorem 5.2]\label{theoVI3} 
Let $i,j$ be two non-negative integers such that $i+j+r\leq q$. One has the two following points:
\begin{itemize}
\item  the image
of the $K\cap \Hn$- equivariant inclusion
$$V^H_{i,j}\subset \bigwedge^{ip,jp}\p\cap\h\subset \bigwedge^{ip,jp}\p$$
 is contained in $V^G_{i,j}$ and
\item the
$K\cap\Hn$-equivariant inclusion $V^H_{i,j}\subset V^G_{i,j}$ 
can be lifted to an $\Hn$-equivariant inclusion
$$A_{i,j}^H\rightarrow A_{i,j}^G.$$
\end{itemize}
\end{theo}

This restriction theorem can be proved using the archimedean theta correspondance and the theory of seesaw pairs
\cite{K1}.

Let $\r$ be the orthogonal complement of $\h$ in $\g$ for the Killing
form. It is an $\h$-module and one has the decomposition
$$\p=\h\cap \p\oplus \q\cap\r.$$
Let $\om_H\in\bigwedge^{2pr}\p\cap \r$, be the vector obtained by taking
the exterior product of a direct orthonormal basis of
$\p_0\cap\r_0$. According to Tong and Wang, one has the following lemma.

\begin{lemm}[\cite{TW} Proposition 4.6]\label{lemmVI2}
The image $\om_H^\text{prim}$ of $\om_H$ by the projection
$$\bigwedge^{2pr}\p\cap \r\rightarrow \bigwedge^{2pr}\p\rightarrow
 V^G_{r,r}$$ 
is non zero.
\end{lemm}
 The theorem \ref{theoVI3} was a theorem on restriction to $H$ of
some $G$-modules. In \cite{B2} is also studied the problem of the
restriction to the diagonal of a tensor product of cohomological
representations.
\begin{lemm}[\cite{B2} lemma 3.15]\label{lemmVI3} 
The image of $V^H_{i,j}\otimes \bigwedge^{2rp}\p\cap\r$, by the inclusion
$$\bigwedge^{ip,jp}\p\cap\h\otimes \bigwedge^{rp,rp}\p\cap
\r\subset
\bigwedge^{(i+r)p,(j+r)p}\p,$$
  is contained in $V^G_{i+r,j+r}.$
\end{lemm}
A somewhat new lemma on the exterior algebra of $\p$ is the following.

\begin{lemm}\label{lemmVI4}
Let $i,j$ be two non-negative integers such that $i+j+r\leq q$. Considering
the application

$$T:\bigwedge^{ip,jp}\p\otimes
\bigwedge^{rp,rp}\p\stackrel{\wedge}{\rightarrow}
\bigwedge^{(i+r)p,(j+r)p}\p{\rightarrow} V^G_{i+r,j+r}$$
then
$$T(V^G_{i,j}\otimes\bigwedge^{2rp}\p)=T(V^G_{i,j}\otimes V^G_{r,r}).$$
\end{lemm}
\begin{proof}
Considering the projections of $V_{i,j}^G\otimes\bigwedge^{*,0}\p$ and
$V_{i,j}^G\otimes\bigwedge^{0,*}\p$ on $V_{i+r,j+r}^G$ as in the proof of
theorem 29 (see in particular equation (39)) of Nicolas Bergeron's article
\cite{B4}, one can obtain the lemma.
\end{proof}

Finally, one has the following theorem (the first point can be
deduced from the lemma \ref{lemmVI3} and \ref{lemmVI4}).

\begin{theo}[\cite{B2} theorem 5.8]\label{theoVI4} Let $i,j$ be two natural integers such that
 $i+j+r\leq q$. One has the two following properties:
\begin{itemize}
\item 
 the application

$$V^G_{i,j}\otimes V^G_{r,r}\subset \bigwedge^{ip,jp}\p\otimes
\bigwedge^{rp,rp}\p\rightarrow \bigwedge^{(i+r)p,(j+r)p}\p{\rightarrow} V^G_{i+r,j+r}$$
is non zero and
\item
 there
exists a non zero orthogonal and $\Gn$-equivariant projection
$$A_{r,r}^G\otimes A_{i,j}^G\rightarrow A_{i+r,j+r}^G,$$
that lifts the projection
$$V^G_{i,j}\otimes V^G_{r,r}\rightarrow V^G_{i+r,j+r}$$
defined in the previous point.
\end{itemize}
\end{theo}

\section{Geometry}\label{sect3}
\subsection{Matsushima formula}
One considers in this paragraph the case of $\Gn=U(a,b)$. Let $\Omega$ be
the Casimir element of $\g$, it is an element of the center
of the envelopping algebra of $\g$. Let $\Ga$ be a discrete subgroup of
$\Gn$  that acts freely and properly on $X_G$. 
\begin{defi}
Let $M$ be a manifold. We will write $A^i(M)$ for the space of smooth differential forms of degree $i$ on  $M$.
\end{defi}
Using translation by $\Gn$, one can see that
\begin{eqnarray}\label{eq3.1}
A^i(S(\Ga))=\Hom_{K}(\bigwedge^i\p,C^\infty(\Ga\bk\Gn)).
\end{eqnarray}
Let us recall that Kuga's lemma says that the action of the Laplacian on the left hand side of
the equation (\ref{eq3.1}) is the same that minus the action of the Casimir
element on the right hand side of equation (\ref{eq3.1}). Let
$\mathcal{H}_2^*(S(\Ga))$ be the space of square-integrable harmonic
forms of degree $*$. Since an harmonic form on $S(\Ga)$ is closed and smooth, one has a
natural application
$$\mathcal{H}_2^*(S(\Ga))\rightarrow H^*(S(\Ga)).$$
Because of Hodge theory, this application is an isomorphism as soon as the
manifold $S(\Ga)$ is compact.

\begin{defi}
Let $\pi$ be a unitary irreducible representation of $\Gn$, by Schur lemma, the Casimir element acts on
$\pi$ by a constant that one writes $\pi(\Omega)$. 
\end{defi}
One writes $L^2(\Gamma\bk G)$ for the unitary representation of $G$
that consists of square-integrable functions on the quotient $\Ga\bk\Gn$. Let $L^2_d(\Gamma\bk \Gn)$ be the discrete spectrum of $L^2(\Gamma\bk\Gn)$ and
${L^2_d(\Gamma\bk \Gn)}^{\Omega=0}$ be
the part of the discrete spectrum on which the Casimir element acts
trivialy. Because of a result of Tong and Wang  \cite[lemme 3.8]{TW}, a
$L^2$-harmonic differential  form on $S_G(\Gamma)$ takes values in the
discrete spectrum. So

\begin{equation}\label{eq2.3}
\mathcal{H}^*_2(S(\Ga))=\Hom_{K}\left(\bigwedge^i\p, L^2_d(\Gamma\bk\Gn)^{\Omega=0}\right).
\end{equation}

This lead to the following definition.
\begin{defi}
A unitary representation $\pi$ of $\Gn$ is called cohomological if:
\begin{itemize}
\item $\pi(\Omega)=0$ and
\item $\Hom_{K}(\bigwedge^*\p,\pi)\neq 0.$
\end{itemize}
\end{defi}
The representations $A_\q$ are clearly cohomological by theorem
\ref{theo2.0}. Indeed Vogan and Zuckermann proved the following theorem.

\begin{theo}[\cite{VZ} theorem 2.5]
Every cohomological representation of $U(a,b)$ is isomorphic to a representation
$A_{\alpha,\beta}^{U(a,b)}$ for $(\alpha,\beta)$ a certain pair of compatible partitions of
$a\times b$. Furthermore if $(\q,L)$ is a parabolic theta stable algebra then
$$\hom_K(\bigwedge^i\p,A_\q)=\hom_K(\bigwedge^i\p,V(\q))=\hom_{K\cap L}(\bigwedge^{i-R(\q)}\p\cap\l,1).$$
\end{theo}

\begin{defi}\label{defi2.3}

 Let  $\q$ be a parabolic theta stable algebra and
  $(\alpha,\beta)$ be the partitions associated to $\q$. A
  square-integrable harmonic differential form on $S(\Ga)$ of degree
  $R(\q)$ is called highly primitive of type $(\alpha,\beta)$ if it is zero
  on 
$\left(V_{\alpha,\beta}^G\right)^\perp$. In the case where
  $(\alpha,\beta)=(i^a,(b-j)^a)$, one will simply say call these classs highly primitive of type $(i,j)$. One writes $H^{ia,ja}(S(\Ga))_{i,j}$ for
  the space of highly primitive forms of type $(i,j)$. 
\end{defi}

More generally, one can define projection to space of highly primitve cohomology.

\begin{defi}\label{defi2.4}
Let $V\subset \bigwedge^k\p$ be a linear $K$-invariant subspace. Let
$V^\perp$ be the orthogonal complement of $V$ in $\bigwedge^k\p$. One has a
the following decomposition
$$A^k(M)=\Hom_K(V,C^\infty(\Gamma\bk \Gn))\oplus \Hom(V^\perp,C^\infty(\Gamma\bk \Gn)).$$
Let $\eta$ be in $A^k(M)$. We will write $\eta_V$ for the projection of $\eta$ on the first factor. If
$V=V_{i,j}^G$, we will just write  $\eta_{i,j}$.
\end{defi}

\begin{rema}\label{rema2.2}
One can consider the exterior product
\begin{equation}\label{wedge}
\bigwedge^k\p^*\otimes\bigwedge^l\p^*\stackrel{\wedge}{\rightarrow}
\bigwedge^{k+l}\p^*,
\end{equation}
and since $\left(\bigwedge^{k}\p^*\right)^*=\bigwedge^{k}(\p)$ by dualizing
this map, one obtains an application
$$\wedge^*:\bigwedge^{k+l}\p{\rightarrow}\bigwedge^k\p\otimes\bigwedge^l\p.$$
 The cup-product of
two differential forms $\eta$ and $\eta'$ of degree respectively $k$ and
$l$ viewed as elements of
$$\Hom_K(\bigwedge^*\p,C^\infty(\Lambda\bk \Gn))\ *=k\text{ ou }l$$
by equation (\ref{eq3.1}) is the element
$$\eta\wedge\eta'\in\Hom_K(\bigwedge^{k+l}\p,C^\infty(\Lambda\bk \Gn))$$
defined by
$$\eta\wedge\eta'(v)(g)=(\eta\otimes\eta')(\wedge^*v)(g).$$
\end{rema}

\subsection{Proof of theorem \ref{theo0.2}}
Let $\Lambda$ be a discrete cocompact subgroup of $\Hn$. One writes  $F$
for the manifold $S_H(\Lambda)$, $M$ for the manifold
$S_G(\Lambda)$ and $R=(i+j)p$. Let $\eta$ be a harmonic differential form of
degree $R$ on $F$ which is supposed highly primitive of type $(i,j)$ . The differential form $j_*\eta$
(defined by equation (\ref{eq0.1})) is the unique smooth form on $M$ of
type $(i+r,j+r)$ such that: $j_*\eta$ is harmonic,  square-integrable and
such that
 for all form $\psi$ of type $(i+r,j+r)$
\begin{equation*}
\int_{M}j_*\eta\wedge *\psi=\int_{F}\eta\wedge*\psi.
\end{equation*}
As explained in the introduction, we start by lifting the form $\eta$ to a
 form on $M$.

\begin{lemm}\label{lemmrel}
There exists a form $p^*{\eta}$ on $M$ of degree $R$ such that
\begin{itemize}
\item $p^*{\eta}$ is smooth harmonic and highly primitive of type $(i,j)$ 
\item and $\left(p^*{\eta}\right)_{|F}  =  \eta$.

\end{itemize}
\end{lemm}

\begin{proof}
Since the form $\eta$ is highly primitive, it  generates in $L^2(\Lambda\bk \Hn)$
under the action of
$\Hn$ an inclusion $\iota:A_{i,j}^H{\subset} L^{2}(\Lambda\backslash H)$. Let $P$ be the  $\Hn$-equivariant projection
\begin{equation}\label{defP}
A_{i,j}^G\rightarrow A_{i,j}^H,
\end{equation}
given by the second point of theorem \ref{theoVI3}. Let $A$ be the action of $\Gn$ on the
representation $A^G_{i,j}$. Let
$v\in\bigwedge^{R}\mathfrak{p}$ and $g\in \Gn$, the vector $A(g)v$ is a
smooth vector of $A_{i,j}^G$ thus $P(A(g)v)$ is also a smooth vector of 
$A_{i,j}^H$. As a consequence, for all $g\in \Gn$,
$\iota\left(P(A(g)v)\right)$ is a smooth function on $\Lambda\bk \Hn$. We define an element 
$$p^*\eta\in\Hom_{K}\left(V_{i,j}^G, C^\infty(\Lambda\bk \Gn)\right)$$
by the formula
$$p^*{\eta}(v)(g)=P(A(g)v)(1).$$
Using equation (\ref{eq3.1}), one can see that $p^*\eta$  is a smooth
differential form of degree $R$ on $F$. Let us show that $p^*\eta$ verifies the two properties of the
lemma. Since the application
$$\left\{\begin{array}{ccc}
A_{i,j}^{G,\infty}&\rightarrow&C^\infty(\Lambda\bk G)\\
v&\mapsto& P(A(g)v)(1)
\end{array}\right.
$$
is $\Gn$-equivariant and that the functions  $p^*{\eta}(v)$ (for
$v\in\bigwedge^R\p$) are contained
in its image, the action of the Casimir element on these
functions is zero. Then, using Kuga lemma, $p^*\eta$ is harmonic and highly
primitive of type
$(i,j)$. It remains to compute the restriction of $p^*\eta$ to $F$. Let $v\in \bigwedge^R\left(\p\cap\h\right)$ then
\begin{eqnarray*}
p^*{\eta}(v)(h)  =  \iota\left(P(A(h)v)\right)(1)
= \iota(P(v))(h)
 =  \eta(v)(h).
\end{eqnarray*}
That way, the lemma is proven.
\end{proof}

Let $\omega$ be the $L^2$ and harmonic differential form on $M$ of degree
$2pr$ dual to $F$. Let $\psi$ be a $L^2$ and harmonic differential form on
$M$ of degree $2pr$, we have
\begin{eqnarray}\label{cupP}
\int_M\omega\wedge*\psi=\int_F*\psi,
\end{eqnarray}
thus, $\omega$ is left invariant by $\Hn$.

\begin{theo}\label{theo3.1}
One has $j_*\eta=(\omega\wedge p^*{\eta})_{i+r,j+r}$ and $j_*\eta$ is non zero.
\end{theo}

\begin{rema}
This theorem implies the theorem \ref{theo0.2}. 
\end{rema}

\begin{proof}

Let us define 
\begin{equation*}
\varphi=(\omega\wedge p^*{\eta})_{i+r,j+r}
\end{equation*}
One has to prove that the form $\varphi$ is: non zero, square-integrable, harmonic and
equal to $j_*\eta$. By now the only thing we
know about $\varphi$ is that this form is closed.

By lemma \ref{lemmVI4}, we have 
\begin{equation*}
\varphi= (\omega_{r,r}\wedge p^*{\eta})_{i+r,j+r}
\end{equation*}

The form $\omega_{r,r}$ is an element of
$$\Hom_K(V_{r,r}^G,L^2_d(\Gamma\bk \Gn)^{\Omega=0})$$
which is left $\Hn$-invariant.  Since it has been assumed that the space $\Lambda\bk \Hn$ is compact, the form $\omega_{r,r}$ is also an
element of
$$\Hom_K(V^G_{r,r},L^2_d(\Hn\bk \Gn)^{\Omega=0}).$$
Thus it generates under the action of $\Gn$ an inclusion
$$ A_{r,r}^G\subset L^2(\Hn\bk \Gn).$$

Using remark \ref{rema2.2}, one can explain the building of $\varphi$ from the point
of view of representation theory. The form $\varphi$ is obtained by the composition of the two
following applications:  

$$\bigwedge^{2rp+R}\rightarrow V^G_{i+r,j+r}\stackrel{\wedge^*}{\rightarrow}
\bigwedge^{2pr}\p\otimes \bigwedge^{R}\p\rightarrow A_{r,r}^G\otimes
A_{i,j}^G,$$
and
$$
A_{r,r}^G\otimes A_{i,j}^G\subset L^2(\Hn\bk \Gn)\otimes 
A_{i,j}^G\subset \text{ind}_\Hn^\Gn\left( {A_{i,j}^G}_{|\Hn}\right)\stackrel{P}{\rightarrow}
\text{ind}_\Hn^\Gn A_{i,j}^H\subset L^2(\Lambda\bk \Gn).$$

It simply means that the functions in  $L^2(\Lambda\bk \Gn)$ defining
$\varphi$ are linear combinations of functions of the form
$$\omega_{r,r}(v)(g)P(A(g)v')(1),$$
with $v\in V^G_{r,r}$ and $v'\in V^G_{i,j}$ (the projection $P$ is defined
by equation (\ref{defP})). Therefore $\varphi$ is
square-integrable. Dualizing the projection of theorem \ref{theoVI4}, one
finds an inclusion
$$A_{i+r,j+r}^G\subset A_{r,r}^G\otimes A_{i,j}^G,$$
lifting the natural and non zero $K$-equivariant application
$$V^G_{i+r,j+r}\subset
\bigwedge^{R+2pr}\p\stackrel{\wedge^*}{\rightarrow}\bigwedge^R\p
\otimes\bigwedge^{2pr}\p\rightarrow V^G_{i,j}\otimes V^G_{r,r}.$$
Since the second application defining $\varphi$ is $\Gn$-equivariant, one deduces
from Kuga lemma that $\varphi$ is harmonic. 

One can now prove that $\varphi$
is non-zero. An element $\omega_H$ of
the line $\bigwedge^{2pr}\r\cap\p$ was defined in subsection \ref{ssect1.2}
and its projection on $V^G_{r,r}$ was written
$\omega_H^\text{prim}$. Because of \cite[proposition 3.5]{B2}, the element
$$\omega_{r,r}(\omega_H)(1)=\omega_{r,r}(\omega_H^{\text{prim}})(1)$$
is non zero. But by lemma \ref{lemmVI3}, the exterior product induced an inclusion
$$V^H_{i,j}\otimes\bigwedge^{2pq}\p\cap\r\subset
V^G_{i+r,j+r}.$$
Therefore if $v\in V_{i,j}^H$ and $h\in \Hn$
\begin{eqnarray*}
\varphi(\omega_H\otimes v)(h)&=&\left(\omega_{r,r}\wedge p^*\eta\right)(
\omega_H\otimes v)(h)\\
&=& \omega_{r,r}(\om_H^{\text{prim}})(1)\eta(v)(h).
\end{eqnarray*}
Since one can choose $v$ and $h$ such that $\eta(v)(h)\neq 0$, $\varphi$ is
non zero. 

Finally it remains to prove that $\varphi$ is
equal to $j_*\eta$. Let $\psi$ be a form of type
 $V_{i+r,j+r}^G$. Since
$$\int_{M}\varphi\wedge*\psi=\int_{M}\omega\wedge (p^*\eta\wedge *\psi),$$
one could be able to conclude immediatly if the form
 $p^*\eta\wedge*\psi$ was square-integrable but it is not the case. Using
the fact that $\omega$ is $\Hn$-invariant one sees that:
\begin{equation}\label{eqfin}
\begin{split}
\int_{M}\varphi\wedge*\psi&=\int_{\Lambda\bk G}\omega(g)\wedge p^*{\eta}(g)\wedge*\psi(g)(\text{vol})dg\\
&=\int_{\Hn\bk G}\omega(g)\wedge\left[\int_{\Lambda\backslash
\Hn}p^*{\eta}(hg)\wedge *\psi(hg)dh\right] dg\\
&=\frac{1}{\text{vol}(\Lambda\bk\Hn)}\int_M \omega\wedge\int_{\Lambda\backslash
\Hn}p^*{\eta}(h)\wedge *\psi(h)dh.
\end{split}
\end{equation}

\begin{lemm} The differential form on $M$
$$\int_{\Lambda\bk H}p^*{\eta}(h)\wedge *\psi(h)dh$$
is square integrable.
\end{lemm}
\begin{proof}
Let $v\in V^G_{i,j}$ and $v'\in V_{i+r,j+r}^G$, one has to prove that
$$\int_{\Hn\backslash \Gn}\left|\int_{\Lambda\backslash \Hn}P(A(g)v)(hg)\psi(v')(hg)dh\right|^2dg<+\infty$$
 using Cauchy-Schwartz, we find that
\begin{multline}
\int_{\Hn\backslash \Gn}\left|\int_{\Lambda\backslash \Hn}P(A(g)v)(hg)\psi(hg)(v')dh\right|^2dg\\
\leq \int_{\Hn\backslash \Gn}||P(A(gv))||_{\Lambda\bk \Hn}^2\left(\int_{\Lambda\bk \Hn}|\psi(hg)|^2dh\right)dg\\
\leq ||v||^2||\psi||^2,
\end{multline}
\end{proof}

Finally using formulas
(\ref{eqfin}), and the definition of
$\omega$, one finds that
\begin{eqnarray*}
\int_M\varphi\wedge*\psi&=&\frac{1}{\text{vol}(\Lambda\bk\Hn)}\int_{F}\int_{\Lambda\bk
\Hn} p^*{\eta}(h)\wedge *\psi(h)dh\\
&=&\int_F p^*\eta\wedge*\psi\\
&=&\int_F\eta\wedge*\psi.
\end{eqnarray*}
So the theorem \ref{theo3.1} is proven.
\end{proof}

\end{document}